\documentclass[11pt]{article}  

\usepackage{authblk}
\usepackage{pstricks}
\usepackage{amsmath}
\usepackage{amsthm}
\usepackage{amssymb}
\usepackage{mathrsfs}
\usepackage{wasysym}
\usepackage{pst-all}
\usepackage{pst-all}
\usepackage{amsmath,amsthm,amssymb,mathrsfs,latexsym,amsfonts}
\usepackage{graphicx,psfrag,epsfig}
\usepackage[english]{babel}
\usepackage[latin1]{inputenc}
\usepackage{geometry}

\newcounter{parag}[subsection]

\newcounter{paraga}[subsection]
\renewcommand{\theparaga}{{\bf\arabic{paraga}.}}
\newcommand{\paraga}{\medskip \addtocounter{paraga}{1} 
\noindent{\theparaga\ } }

\newcounter{pparag}

\newtheorem{theoreme}{Theorem}[section]

\newtheorem{lemme}[theoreme]{Lemma}
\newtheorem{prop}[theoreme]{Proposition}

\newtheorem{cor}[theoreme]{Corollary}
\newtheorem{definition}[theoreme]{Definition\rm}
\newtheorem{remarque}[theoreme]{\it Remark}

\def\Bbibitem#1#2{\bibitem[#1]{#2}}

\def\text#1{\,\hbox{#1}\;}

\def\text#1{\,\hbox{#1}\;}

\def\al{\alpha}

\def\ga{\gamma}
\def\Ga{{\Gamma}}
\def\de{\delta}

\def\De{\Delta}
\def\eps{{\varepsilon}}
\def\ka{\kappa}
\def\la{\lambda}

\def\om{\omega}
\def\Om{\Omega}

\def\sig{{\sigma}}
\def\Sig{{\Sigma}}

\def\th{{\theta}}

\def\ph{\varphi}

\def\jA{{\mathscr A}}

\def\jC{{\mathscr C}}
\def\jD{{\mathscr D}}

\def\jN{{\mathscr N}}
\def\jO{{\mathscr O}}

\def\jT{{\mathscr T}}
\def\jU{{\mathscr U}}
\def\jV{{\mathscr V}}
\def\jg_\eps{{\mathscr g_\eps}}

\def\bC{{\bf C}}

\def\A{{\mathbb A}}

\def\R{{\mathbb R}}
\def\S{{\mathbb S}}
\def\T{{\mathbb T}}
\def\Z{{\mathbb Z}}
\def\I{{\mathbb I}}

\def\Max{\mathop{\rm Max\,}\limits}
\def\Min{\mathop{\rm Min\,}\limits}

\def\setm{\setminus}

\def\ov{\overline}
\def\til{\widetilde}
\def\ha{\widehat}
\def\d{\partial}
\def\inv{^{-1}}

\def\pdemi{{\tfrac{1}{2}}}
\def\abs#1{\left\vert#1\right\vert}
\def\norm#1{\Vert#1\Vert}
\def\setm{\setminus}

\def\sA{{\mathsf A}}

\def\e{{\mathbf e}}
\def\bu{{\bullet}}
\def\HH{{\bf H}}
\def\bA{{\bf A}}
\def\cA{{\mathcal A}}
\def\dd{{\bf d}}
\def\H{{\mathcal H}}

\def\beq{\begin{equation}}
\def\eeq{\end{equation}}

\def\cyl{{\rm Cyl}}
\def\chain{{\rm Chain}}
\def\tori{{\rm Tori}}

\begin{titlepage}
\author[1]{Jean-Pierre Marco}
\author[2]{Lara Sabbagh \thanks{This work was partially supported by the EPSRC grant EP/J003948/1.}}
\title{\LARGE{\textbf{Examples of nearly integrable systems on $\A^3$ with asymptotically dense projected orbits}}}
\affil[1]{IMJ, University Paris VI, email: marco@math.jussieu.fr}
\affil[2]{Mathematics Institute, University of Warwick, email: l.el-sabbagh@warwick.ac.uk}
\end{titlepage}
\begin{document}
\maketitle

%
%
%
%
%

\vskip10mm

\begin{abstract} Given an integer $\ka\geq2$, we introduce a class of nearly integrable 
systems on $\A^3$, of the form
$$
H_n(\th,r)=\pdemi \norm{r}^2+\tfrac{1}{n} U(\th_2,\th_3)+f_n(\th,r)
$$
where $U\in C^\ka(\T^2)$ is a generic potential function and $f_n$ a $C^{\ka-1}$ additional
perturbation such that $\norm{f_n}_{C^{\ka-1}(\A^3)}\leq \tfrac{1}{n}$, so that $H_n$ is a
perturbation of the completely integrable system $h(r)=\pdemi\norm{r}^2$. 
 
Let $\Pi:\A^3\to\R^3$ be the canonical projection.  We prove that for each $\de>0$, there exists 
$n_0$ such that for $n\geq n_0$, the system $H_n$ admits an orbit $\Ga_n$ at energy $\pdemi$
whose projection $\Pi(\Ga_n)$ is $\de$--dense  in $\Pi(H_n\inv(\pdemi))$, in the sense that the 
$\de$--neighborhood of $\Pi(\Ga_n)$ in $\R^3$ covers $\Pi(H_n\inv(\pdemi))$.
\end{abstract}

\vskip25mm


\section{Introduction and main result}
The aim of this paper is to construct a simple class of {\em a priori} stable nearly integrable systems on $\A^3$ 
for which the dynamical behavior caused by a double resonance plays the central role and yields the existence of 
``asymptotically dense projected orbits", that is, orbits at fixed energy whose projection on the energy level
passes within an arbitrarily small distance from each point of the projected energy level, when the size of
the perturbation tends to $0$.

\vskip1mm

Several more general results on Arnold diffusion were recently announced. The goal of this paper is more
modest, we try to underline the geometry of the diffusion process and to get rid of all (heavy) technical details
as far as possible. We however think that our present work can help understand the more sophisticated
methods coming into play for the proof of the so-called Arnold conjecture. We refer to the forthcoming
publications by C.-Q. Cheng, V. Kaloshin and Ke Zhang, and J.-P. Marco for the complete studies
(see \cite{C12,Mar1,Mar2,KZ12}) .

\vskip1mm

Given an integer $m\geq 1$, we denote by $\A^m=\T^m\times\R^m$ the cotangent bundle of the torus $\T^m$ 
that we endow with its usual angle-action coordinates $(\th,r)$ and its Liouville symplectic form  
$\Om=\sum_{i=1}^mdr_i\wedge d\th_i$.   We denote by $\Pi$ the projection $\A^m\to\R^m$. 
When $H$ is a $C^\ka$ function on an open set of $\A^m$, $\ka\geq 2$, 
we denote by $X^H$ its Hamiltonian vector field and by $\Phi_t^{H}$ its local flow.
Given a function $H$ and an element $a$ in its
range, we write $H\inv(a)$ instead of $H\inv(\{a\})$, even if $H$ is not a bijection.

\vskip1mm

Our systems will be defined on $\A^3$ and have the following form
$$
H_n(\th,r)=\pdemi \norm{r}^2+\tfrac{1}{n} U(\th_2,\th_3)+f_n(\th,r),
$$
where $\norm{\cdot}$ stands for the Euclidean norm, 
$U\in C^\ka(\T^2)$ is a generic potential function and $f_n\in C^{\ka-1}(\A^3)$ is an additional perturbation 
such that $\norm{f_n}_{C^{\ka-1}(\A^3)}\leq \tfrac{1}{n}$.  For the sake of simplicity, we limit ourselves to the case 
where $\ka$ is an integer $\geq 2$ but the construction could easily be extended to the $C^{\infty}$ or Gevrey cases
as well. 

\vskip1mm

The system $H_n$ is a perturbation of the integrable system $h(r)=\pdemi\norm{r}^2$.
We will focus on the energy level $H_n\inv(\pdemi)$ but  any other positive energy level would have the same 
properties. The frequency map associated with
$h$ is 
$$
\om(r)=r,
$$
and the double resonance under concern is the set of actions $r$ such that $\om_2(r)=\om_3(r)=0$, that is,
the line $r_2=r_3=0$. This line intersects the unperturbed level $\S=h\inv(\pdemi)$ (the unit sphere)
at the points $D_\pm=(\pm 1,0,0)$. Both averaged systems at these points have the same ``principal part'', namely:
$$
\ov H_n(\th_2,\th_3,r_2,r_3)=\pdemi\big(r_2^2+r_3^2)+\tfrac{1}{n} U(\th_2,\th_3).
$$
The full averaged systems also contain the average of $f_n$, but this will be insignificant thanks to a proper choice
of this additional perturbation.
The system $\ov H_n$ is of ``classical form'', the sum of a kinetic part and a potential part.  The potential $U$
is the main data of the problem, it will be arbitrarily chosen in a residual subset of $C^\ka(\T^2)$.
In particular, the system $\ov H_n$ will be 
nonintegrable. This property is in contrast with the previous studies on double resonances where the averaged system
was usually assumed to be integrable or nearly integrable (see \cite{Bes97}). However, this nonintegrability and
the associated ``chaotic behavior'' are essential features of generic  nearly integrable systems as proved in the recent 
studies on Arnold diffusion. 
On the contrary,  the last term $f_n$ of the perturbation will be a ``very  nongeneric''  bump function, 
especially designed to  easily create and control the  so-called ``splitting of separatrices'' in the spirit of 
\cite{D88,MS02}.

\vskip1mm

The truncated system
\beq\label{eq:truncham}
\H_n(\th,r)=\pdemi\norm{r}^2+\tfrac{1}{n} U(\th_2,\th_3),\qquad (\th,r)\in\A^3,
\eeq
does not admit diffusion orbits. In fact, it appears as the direct product of the one-degree-of-freedom Hamiltonian
$
\pdemi r_1^2
$
with the previous system $\ov H_n$, and the conservation of energy in both factors prevents from any
diffusion phenomenon. It is only when the perturbation $f_n$ is added that the splitting of separatrices appears
and makes the diffusion possible. The structure of our system is therefore in some sense analogous to that of Arnold's 
initial model for diffusion  along a simple resonance (\cite{A64}). But while in Arnold's model the diffusion phenomenon 
occurs only along  a single resonance, in our  model the diffusion takes place along a very large family of simple 
resonances,  namely the great circles of $\S$ orthogonal to the vectors $k=(0,k_2,k_3)$, where 
$k_2,k_3$ are  coprime integers. The previous  double resonant points $D_\pm$ are the places where exchanges 
of resonances are  made possible by the structure of the averaged systems in their neighborhood. 

\vskip1mm

Let us now state our main result.
For $2\leq \ka< +\infty$, we endow the spaces $C^\ka(\T^2)$ of $C^\ka$ functions on $\T^2$ with their usual 
$C^\ka$ norms
$$
\norm{U}_{C^\ka(\T^2)}=\Max_{\abs{\al}\leq \ka} \Max_{\th\in\T^2}\abs{\partial ^\al U(\th)},
$$
which make them Banach spaces.  Throughout this paper, the triples $x=(x_1,x_2,x_3)$ in $\T^3$ or $\R^3$ 
will also be denoted by 
$$
x=(x_1,\ov x), \qquad \ov x=(x_2,x_3).
$$
We also introduce a formal definition for the notion of ``approximative density".
Given a metric space $(E,d)$ and $\de>0$, we say that a subset
$S$ of $E$ is $\de$-dense in a subset $F\subset E$ when $F$ is contained in the union of the family of all open 
$\de$-balls centered at points of $S$. We will prove the following diffusion result.

\begin{theoreme}\label{thm:main} 
Let $\ka\geq 2$ be a fixed integer. Then
there exists a residual subset $\jU$ in $C^\ka(\T^2)$ such that 
for each $U\in \jU$, there exists a sequence $(f_n)_{n\geq 1}$ of $C^{\ka-1}$ functions on $\A^3$,
with $\norm{f_n}_ {C^{\ka-1}(\A^3)}\leq\tfrac{1}{n}$,
such that for any $\de>0$, there exists $n_0$ such that for $n\geq n_0$, the system
\begin{equation}
H_n(\th,r)=\pdemi \norm{r}^2 +\tfrac{1}{n} U(\ov \th)+f_n(\th,r),\qquad (\th,r)\in\A^3,
\end{equation}
admits an orbit $\Ga_n$ with energy $\pdemi$ such that $\Pi(\Ga_n)$ is $\de$--dense in $\Pi(H_n\inv(\pdemi))$.
\end{theoreme}

Since we only aim at producing examples, the fact that $\jU$ is nonempty would be enough. The fact that
the set of ``convenient'' potentials is residual was proved in \cite{Mar1} but of course a single potential
would be enough to construct an example. Since examples are not too difficult to produce, the content
of this paper can be seen as independent from \cite{Mar1}.
However, taking for granted the residual character of $\jU$ makes it plausible --eventhough we do not try to prove this--
that the wild behavior of orbits described in our examples is in fact ``typical'' for a priori stable perturbations
of integrable systems. As mentionend above, this last question has been recently investigated by several authors.

\vskip1mm

We could also work in the class of diffeomorphisms, in which case an analogous 
construction yields examples of nearly-integrable diffeomorphisms with a large class of orbits biasymptotic to 
infinity.  However we will limit ourselves here to the Hamiltonian case, which is indeed richer and sligthly more difficult
due to the additional geometrical difficulty induced by the preservation of energy. 
The paper is organized as follows.

\vskip1mm

$\bullet$ Section 2 is devoted to the description of those properties of  classical systems which will be needed to
construct our examples, namely the existence of suitable {\em chains of annuli}. Here we summarize \cite{Mar1}. 
Again, we emphasize that our present construction is to a large extent independent of this latter work 
(apart from the necessary definitions), the concern of which is the  genericity of the potential $U$ for which
the associated classical system possesses suitable chains of annuli.

\vskip1mm

$\bullet$ In Section 3, we deduce from the previous properties of classical systems the existence of {\em chains
of cylinders} in our systems $\H_n$, and we prove that these chains project in the space of actions asymptotically
close to a dense family of great circles in the unit sphere (the simple resonance lines). These cylinders are normally 
hyperbolic invariant manifolds diffeomorphic to $\T^2\times\,[0,1]$ and admit a foliation by invariant tori diffeomorphic 
to $\T^2$.

\vskip1mm

$\bullet$ In Section 4, we construct the sequence $(f_n)$ in such a way that each invariant torus in the previous
family admits a homoclinic orbit along which its stable and unstable manifolds intersect transversely in a weak
sense. This in particular yields the existence of heteroclinic connections between nearby enough tori contained in the
same cylinders. Other  transversality properties for heteroclinic orbits between tori belonging to distinct cylinders of the chains 
are also proved.

\vskip1mm

$\bullet$ Finally in Section 5, we prove the existence of the diffusion orbits. The key result there is 
the $\la$--lemma proved in \cite{S13} which is specially designed for normally hyperbolic manifolds and which 
enables us to prove very easily the necessary shadowing results.


\section{Classical systems}

This section is devoted to the description of the generic hyperbolic properties of classical systems on the torus
$\T^2$ which will be needed in the construction of our examples. Given a potential function $U\in C^\ka(\T^2)$,
we define here the associated classical system as the Hamiltonian on $\A^2$
\begin{equation}\label{eq:classham}
 C_U(x,y)=\pdemi \norm{y}^2+ U(x),
\end{equation}
where $x\in\T^2$ and $y\in\R^2$.
We will always require the potential $U$ to admit a single maximum $\ha e$ at some $x^0$, which is  nondegenerate
in the sense that the Hessian of $U$ is negative definite.  This is of course true for a $U$ in a residual subset
$\jU_0\subset C^\ka(\T^2)$. It is then easy to check that the lift of $x^0$ to the zero section of $\A^2$ is a 
hyperbolic fixed point for $X^{C_U}$.


\paraga We denote by $\pi:\A^2\to\T^2$  the canonical projection and we fix $U\in\jU_0$ together with the 
associated classical system $C:=C_U$.

\begin{definition}\label{def:annuli}
Let $c\in H_1(\T^2,\Z)$ and let $I\subset\R$ be an interval. 
An {\em annulus for $X^C$  realizing $c$ and defined over $I$} is a submanifold $\sA$, contained 
in $C\inv(I)\subset\A^2$, such that
\begin{itemize}
\item for each $e\in I$,  $\sA\cap C\inv(e)$ is the orbit of a periodic solution $\ga_e$ of $X^C$, 
which is hyperbolic in $C\inv(e)$, with orientable stable and unstable bundles, and such that the projection 
$\pi\circ\ga_e$ on $\T^2$ realizes $c$;
\item when $e>\ov e$, the frequency $\om(e)$ of the solution $\ga_e$ is an increasing function of $e$
and when $e<\ov e$, $\om(e)$ is a decreasing function of $e$;
\item there exists a covering $I=\cup_{1\leq i\leq i^*} I_i^*$ of $I$ by open subintervals of $I$ such that 
for $1\leq i\leq i^*$ and for $e\in I^*_i$, the solution $\ga_e$ admits a homoclinic solution $\om^i_e$ 
along which the stable and unstable manifolds of $\ga_e$  intersect transversely inside $C\inv(e)$.
\item one can choose the covering $I=\cup_{1\leq i\leq i^*} I_i^*$ in such a way that $I_i$ and $I_j$
are disjoint if $\abs{j-i}\geq 2$ and the solutions $\om^i_e$ and $\om^{i+1}_e$ are geometrically
disjoint for $e$ in the intersection $I^*_i\cap I^*_{i+1}$.
\end{itemize}
\end{definition}

Since the solutions $\ga_e$ are hyperbolic and vary continuously with $e$ (since $\sA$ is assumed to 
be a submanifold),  the annulus  $\sA$ is a $C^{\ka-1}$ submanifold of $\A^2$, with boundary 
$\d\sA\sim \T\times \d I$.  It is clearly normally hyperbolic (the boundary causes no trouble is this 
simple setting, due to the conservation of the Hamiltonian), and its stable and unstable manifolds are the 
unions of those of the periodic solutions  $\ga_e$.  Note that when $I$ has a boundary point, the family 
$\ga_e$ can be continued  over a slightly larger open interval, but it will be interesting to allow the intervals
to be compact in our subsequent constructions. 

\vskip2mm

It is not difficult to prove that there exists an
embedding $\phi:\T\times I\to \A^2$  whose image is $\sA$ and which satisfies  
\begin{equation}\label{eq:embedd}
C\circ\phi(\ph,e)=e.
\end{equation}
Note that obviously
$\phi(\T\times\{e\})=\sA\cap C\inv(e)$.
Moreover, one can find a symplectic embedding $\phi_s$ such that $C\circ\phi_s$ is in action-angle form
(that is, does not depend on the angle),
where of course $\T\times I$ is equipped with its usual symplectic structure. 


\paraga Due to the reversibility of $C$, the solutions of the vector field $X^C$ occur in {\em opposite pairs} 
(pairs of symmetric solutions whose time parametrizations are exchanged by the symmetry $t\mapsto-t$). 
This is in particular the case for the solutions homoclinic to the hyperbolic fixed point $O$ associated with 
the maximum $x^0$ of $U$. We set 
$$
\ha e=\Max U=U(x^0).
$$

\begin{definition}\label{def:singannuli} Let $c\in H_1(\T^2,\Z)\setm\{0\}$.
A {\em singular annulus for $X^C$  realizing $\pm c$, with parameters $\til e>\ha e$ and 
$e^0<\ha e$}, is a $C^1$ invariant manifold $\sA_\bu$ with boundary , diffeomorphic to the sphere 
$S^2$ minus three disjoint open discs,  such that, setting $I=\,]\ha e,\til e]$ and $I_0=[e^0,\ha e[$:
\begin{itemize}
\item $\sA_\bu\cap\, C\inv(\ha e)$  is the union of the hyperbolic  fixed point  $O$ and a pair of opposite 
homoclinic orbits,
\item  $\sA_\bu\cap C\inv(I)$ admits two connected components $\sA_\bu^>$ and $\sA_\bu^<$, 
which are annuli  defined over $I$ and  realizing $c$ and $-c$ respectively,
\item  $\sA_\bu^0=\sA_\bu\cap C\inv(I_0)$ is an annulus defined over $I_0$ and realizing the null class $0$,
\item $\sA_\bu$ admits a $C^1$ stable (resp. unstable) manifold, in which the union of the stable (resp. unstable) 
manifolds of $\sA_\bu^>$, $\sA_\bu^<$ and $\sA_\bu^0$ is dense.
\item both homoclinic orbits admit  homoclinic connections along which the stable and unstable manifolds
of $\sA_\bu$ intersect transversely in $C\inv(\ha e)$.
\end{itemize}
\end{definition}

Note that a singular annulus $\sA_\bu$ is ``almost everywhere $C^{\ka-1}$'', since the  connected 
components of 
$\sA_\bu\cap C\inv(I)$ and $\sA_\bu\cap C\inv(I_0)$ are annuli, so $C^{\ka-1}$ submanifolds of $\A^2$. 
Note also that $\sA_\bu$ is a center manifold for both homoclinic orbits in $\sA_\bu\cap C\inv(\ha e)$, with 
hyperbolic  transverse spectrum. One can in fact prove that a singular annulus is slightly more regular 
than $C^1$  (depending on the Lyapunov exponents of the fixed point $O$), but this is useless here. 

\vskip2mm 

A singular annulus is depicted in Figure \ref{Fig:singcyl}: it is essentially 
the part of the phase space of  a simple pendulum limited by two essential invariant curves at the same energy,
 from which a neighborhood of the elliptic fixed point was 
removed. More precisely, on the annulus $\A$ equipped with the coordinates $(\ph,I)$, we define a
``pendulum Hamiltonian'' as a Hamiltonian of the form
$$
P_{\ha e}(\ph,I)=\pdemi I^2 +V(\ph)+\ha e
$$ 
where $V$ is a $C^2$ potential function with a single nondegenerate maximum at $0$ and a single nondegenerate
minimum, which satisfies $V(0)=0$. 
For $a<\ha e<b$ we introduce the subset $\cA_\bu(a,b)$ 
defined by  $a \leq P_{\ha e}(\ph,I) \leq b$.  So $\cA_\bu(a,b)$ is the zone bounded by the two 
invariant curves of equation $P_{\ha e}=b$, together with an invariant curve surrounding  the elliptic point. We call 
$\cA_\bu(\ha e,V,a,b)$  the {\em standard singular annulus} with parameters $(\ha e,V,a,b)$. One  
proves  that a singular  annulus is $C^1$  diffeomorphic  to some standard annulus by a 
diffeomorphism $\phi_\bu:\cA_\bu(a,b)\to\sA_\bu$ such that 
\begin{equation}\label{eq:fibu}
C_{\vert\sA_\bu}\circ \phi_\bu=P_{\ha e}.
\end{equation}

\begin{remarque}\label{rem:emb}
By defintion of an annulus, there exist embeddings $\phi_\bu^>:\T\times \,]\ha e,\til e\,]\to \sA_\bu^>$,
$\phi_\bu^<:\T\times \,]\ha e,\til e\,]\to \sA_\bu^<$
and $\phi^0_\bu:\T\times [e^0,\ha e[\,\to \sA_\bu^0$  for the 3 subannuli of a singular annulus.
These embeddings satisfy
\beq\label{eq:embedsing}
C\circ\phi_\bu^*(\ph,e)=e,
\eeq
where $*$ stands for $>,<$ or $0$.
\end{remarque}

\vskip-4mm 
\begin{figure}[h]
\begin{center}
\begin{pspicture}(0cm,2.5cm)
\psset{xunit=.5cm,yunit=.4cm}
\rput(-4,0){
\psellipse(8,2.5)(.5,1.5)
\psellipse(0,2.5)(.5,1.5)
\psframe[linestyle=none,fillstyle=solid,fillcolor=white](0,0)(.8,4)
\pscurve(3.5,3)(3.8,2.2)(4.5,3)
\pscurve(3.5,3)(3.8,2.55)(4.05,2.3)
\pscurve(0,4)(2,3.2)(3.5,3)
\pscurve(4.5,3)(6,3.2)(8,4)
\pscurve(4.5,3)(6,3.2)(8,4)
\pscurve(0,1)(2,.2)(4,-.1)(6,.2)(8,1)
\pscircle[fillstyle=solid,fillcolor=black,linecolor=black](4,-.1){.05}
\rput(4,-.7){$O$}
\pscurve[linecolor=black](4,-.1)(3.5,.8)(3.2,1.6)(3.1,2.5)(3.2,3)
\pscurve[linecolor=black](4,-.1)(4.5,.8)(4.8,1.6)(4.9,2.5)(4.8,3)
\rput(8,5){$\sA_\bu^>$}
\psline[linewidth=.1mm](8,4.6)(7,3)
\rput(0,5){$\sA_\bu^<$}
\psline[linewidth=.1mm](0,4.4)(1,3)
\rput(4,5){$\sA_\bu^0$}
\psline[linewidth=.1mm](4,4.3)(4,1.5)
\rput(1.5,-2){\color{black}$C\inv(\ha e)\cap\sA_\bu$}
\psline[linewidth=.1mm](2,-1.3)(3.5,.9)
\psline[linewidth=.1mm](2,-1.3)(4.5,.9)
}
\end{pspicture}

\vskip8mm
\caption{A singular annulus}\label{Fig:singcyl}
\end{center}
\end{figure}
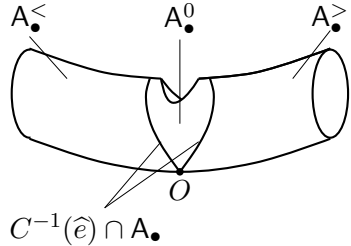

\vspace{-3mm}


\paraga Let us now turn to the definition of chains of annuli for the classical system $C$.
We say  that a family $(I_i)_{1\leq i\leq m}$ of  nontrivial closed subintervals of $]\ha e,+\infty[$ is 
{\em ordered} when  $\Max I_i>\Min I_{i+1}$ for $1\leq i\leq m-1$. Of course in the following we will
assume that $\Max I_i$ is only slightly larger that $\Min I_{i+1}$, so that the overlapping between
two consecutive intervals is only a small neighborhood of their extremities.

\begin{definition}\label{def:chains} Let $c,c'\in H_1(\T^2,\Z)\setm\{0\}$.
\begin{itemize}
\item A {\em chain of annuli realizing $c$} is a  family $(\sA_i)_{0\leq i\leq m}$ of 
annuli realizing $c$, defined over an ordered family $(I_i)_{0\leq i\leq m}$ of closed subintervals of 
$]\ha e,+\infty[$, with the additional property that for $0\leq i\leq m-1$, if $\Ga_i(e)$ stands for the hyperbolic
orbit at energy $e$ on $\sA_i$, if $e\in[\Min I_{i+1},\Max I_i]$:
$$
W^u\big(\Ga_i(e)\big)\cap W^s\big(\Ga_{i+1}(e)\big)\neq \emptyset,\qquad 
W^s\big(\Ga_i(e)\big)\cap W^u\big(\Ga_{i+1}(e)\big)\neq \emptyset,
$$
the intersection being transverse in the corresponding energy level $C\inv(e)$.
\item A {\em generalized chain of annuli realizing $c$ and $c'$}
is the union of two chains $(\sA_i)_{0\leq i\leq m}$ and $(\sA'_i)_{0\leq i\leq m'}$  realizing $c$ and $c'$ 
respectively, together with a singular annulus $\sA_\bu$, such that 
$$
W^u(\sA_{0})\cap W^s(\sA_\bu)\neq \emptyset,\qquad 
W^s(\sA_{0})\cap W^u(\sA^{\bu})\neq \emptyset,
$$
$$
W^u(\sA'_{0})\cap W^s(\sA_\bu)\neq \emptyset,\qquad 
W^s(\sA'_{0})\cap W^u(\sA^{\bu})\neq \emptyset,
$$
the intersections being transverse in $\A^2$.
\end{itemize}
\end{definition}

Note that we do not specify the homology of the singular annulus $\sA_\bu$, this latter turns out to be 
fixed independently of the classes $c$ and $c'$ in our subsequent construction.


\paraga We  now state the genericity result from \cite{Mar1}. 
We say that $c\in H_1(\T^2,\Z)\setm\{0\}$ is primitive when the 
equality $c=mc'$ with $m\in\Z$ implies $m=\pm1$.  We denote by $\HH_1(\T^2,\Z)$ the set of primitive 
homology classes. Let $\dd$ be the Hausdorff distance for compact subsets of $\R^2$ and  
$\Pi:\A^2\to\R^2$ the canonical projection. Recall that $\jU_0$ is the set of potentials with a single nondegenerate
maximum.

\vskip3mm

\begin{theoreme} {\bf (\cite{Mar1})}. \label{thm:genprop}
Let $2\leq \ka\leq +\infty$.
Then there exists a  residual subset $\jU\subset\jU_0$ in $C^\ka(\T^2)$ such that for $U\in \jU$,  the 
associated classical  system $C_U$  defined in (\ref{eq:classham}) satisfies the following properties.
\begin{enumerate}
\item For each  $c\in \HH_1(\T^2,\Z)$ there exists a chain $\bA(c)=(\sA_0,\ldots,\sA_m)$ of
annuli realizing~$c$, defined over ordered intervals $I_0,\ldots,I_m$, such that the first and last intervals 
are of the form
$$
I_0=\,]\Max U,e_m]\quad \textit{and}\quad I_m=[e_P,+\infty[,
$$
for suitable constants $e_m$ and $e_P$.

\item Let $c=(c_1,c_2)$ be the canonical identification of $H_1(\T^2,\Z)$ with $\Z^2$ and for $e>0$,
set 
$$
Y_c(e)=\frac{\sqrt {2e}\,c}{\norm{c}}
$$
Setting $\Ga_e=\sA_m\cap C_U\inv(e)$ for $e\in [e_P,+\infty[$, then
$$
\lim_{e\to+\infty}\dd\big(\Pi(\Ga_e),\{Y_c(e)\}\big)=0
$$


\item There exists a singular annulus $\sA_\bu$ which admits transverse heteroclinic connections with every 
first  annulus in the previous chains.
\end{enumerate}
\end{theoreme}

\vskip2mm

The existence of the ``high energy annuli'' $\sA_m$ is proved by a simple argument {\em \`a la} Poincar\'e,
on the creation of hyperbolic orbits near perturbations of resonant tori, so we call $e_P$ the Poincar\'e energy
for the class $c$. The other annuli are proved to exist by minimization arguments of Morse and Hedlund.

\vskip3mm

There exist in general several singular annuli with the previous intersection property, but one will be enough 
for our future needs.  We say that a chain with $I_0$ and $I_m$ as in the first item above is {\em biasymptotic 
to $\ha e:=\Max U$ and $+\infty$}. It may be useful to rephrase  Theorem \ref{thm:genprop} in a more concise way.

\begin{cor} For $U\in\jU$ and for each pair of classes $c,c'\in \HH_1(\T^2,\Z)$, there exists a generalized 
chain of annuli,  union of $(\sA_i)_{0\leq i\leq m}$, $(\sA'_i)_{0\leq i\leq m'}$
and $\sA_\bu$, such that $(\sA_i)_{0\leq i\leq m}$ and $(\sA'_i)_{0\leq i\leq m'}$ are biasymptotic to $\ha e$
and $+\infty$ and realize $c$ and $c'$ respectively.
\end{cor}

In the $y$--plane, one therefore gets the following picture for the projection of generalized chains of annuli,
along some lines of rational slope (which obviously correspond to integer homology classes when the energy tends
to $+\infty$, by Theorem \ref{thm:genprop}, 2).

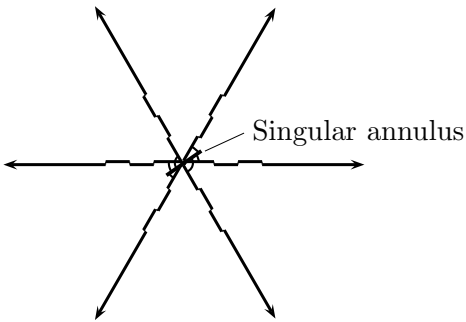
\begin{figure}[h]
\begin{center}
\begin{pspicture}(0cm,2.3cm)
\psset{unit=.4cm}
\rput{-10}{
\psline[linewidth=0.5mm](-.5,-.5)(.5,.5)
}
\psline[linewidth=0.1mm](.7,.4)(2,1)
\rput(5.8,1){Singular annulus}
\pscircle[fillstyle=solid,fillcolor=white](0,0){.07} 
\psline[linewidth=0.4mm](0.06,0.05)(1,0.05)
\psline[linewidth=0.1mm](1,0.05)(1,-0.05)
\psline[linewidth=0.4mm](1,-0.05)(1.8,-0.05)
\psline[linewidth=0.1mm](1.8,0.05)(1.8,-0.05)
\psline[linewidth=0.4mm](1.8,0.05)(2.6,0.05)
\psline[linewidth=0.1mm](2.6,0.05)(2.6,-0.05)
\psline[linewidth=0.4mm]{->}(2.6,-0.05)(6,-0.05)
\psline[linewidth=0.4mm](-0.06,0.05)(-1,0.05)
\psline[linewidth=0.1mm](-1,0.05)(-1,-0.05)
\psline[linewidth=0.4mm](-1,-0.05)(-1.8,-0.05)
\psline[linewidth=0.1mm](-1.8,0.05)(-1.8,-0.05)
\psline[linewidth=0.4mm](-1.8,0.05)(-2.6,0.05)
\psline[linewidth=0.1mm](-2.6,0.05)(-2.6,-0.05)
\psline[linewidth=0.4mm]{->}(-2.6,-0.05)(-6,-0.05)
\rput{120}{
\pscircle[fillstyle=solid,fillcolor=white](0,0){.07} 
\psline[linewidth=0.4mm](0.06,0.05)(1,0.05)
\psline[linewidth=0.1mm](1,0.05)(1,-0.05)
\psline[linewidth=0.4mm](1,-0.05)(1.8,-0.05)
\psline[linewidth=0.1mm](1.8,0.05)(1.8,-0.05)
\psline[linewidth=0.4mm](1.8,0.05)(2.6,0.05)
\psline[linewidth=0.1mm](2.6,0.05)(2.6,-0.05)
\psline[linewidth=0.4mm]{->}(2.6,-0.05)(6,-0.05)
\psline[linewidth=0.4mm](-0.06,0.05)(-1,0.05)
\psline[linewidth=0.1mm](-1,0.05)(-1,-0.05)
\psline[linewidth=0.4mm](-1,-0.05)(-1.8,-0.05)
\psline[linewidth=0.1mm](-1.8,0.05)(-1.8,-0.05)
\psline[linewidth=0.4mm](-1.8,0.05)(-2.6,0.05)
\psline[linewidth=0.1mm](-2.6,0.05)(-2.6,-0.05)
\psline[linewidth=0.4mm]{->}(-2.6,-0.05)(-6,-0.05)
}
\rput{60}{
\pscircle[fillstyle=solid,fillcolor=white](0,0){.07} 
\psline[linewidth=0.4mm](0.06,0.05)(1,0.05)
\psline[linewidth=0.1mm](1,0.05)(1,-0.05)
\psline[linewidth=0.4mm](1,-0.05)(1.8,-0.05)
\psline[linewidth=0.1mm](1.8,0.05)(1.8,-0.05)
\psline[linewidth=0.4mm](1.8,0.05)(2.6,0.05)
\psline[linewidth=0.1mm](2.6,0.05)(2.6,-0.05)
\psline[linewidth=0.4mm]{->}(2.6,-0.05)(6,-0.05)
\psline[linewidth=0.4mm](-0.06,0.05)(-1,0.05)
\psline[linewidth=0.1mm](-1,0.05)(-1,-0.05)
\psline[linewidth=0.4mm](-1,-0.05)(-1.8,-0.05)
\psline[linewidth=0.1mm](-1.8,0.05)(-1.8,-0.05)
\psline[linewidth=0.4mm](-1.8,0.05)(-2.6,0.05)
\psline[linewidth=0.1mm](-2.6,0.05)(-2.6,-0.05)
\psline[linewidth=0.4mm]{->}(-2.6,-0.05)(-6,-0.05)
}
\psarc(0,0){.6}{33}{66}
\psarc(0,0){.5}{3}{33}
\psarc(0,0){.3}{-66}{33}
\psarc(0,0){.6}{217}{237}
\psarc(0,0){.5}{178}{217}
\psarc(0,0){.3}{130}{217}
\end{pspicture}
\vskip2cm
\caption{Projected generalized chains of annuli in a classical system}
\end{center}
\end{figure}

\vskip-1cm



\section{Chains of cylinders for $\H_n$}
Here we call {\em cylinder} for a vector field defined on the cylinder $\A^3$ a normally hyperbolic invariant 
manifold with boundary, diffeomorphic to $\T^2\times [0,1]$. In particular, the stable and unstable manifolds 
of a cylinder are well-defined and the definition of heteroclinic connections between cylinders can be properly stated.
In this section, we prove the existence of  a family of chains of cylinders for the truncated system $\H_n$
defined in~(\ref{eq:truncham}), in the energy level  $\H_n\inv(\pdemi)$, whose  projection by $\Pi$ forms an 
asymptotically  dense subset of the unit sphere.


\subsection{Cylinders and chains}

\paraga  Let us set out a first definition in a context adapted to our needs (but which can obviously be adapted
to more general ones).

\begin{definition} \label{def:invcyl}  Let $X$ be a vector field on $\A^3$. 

\begin{itemize}
\item We say that $\jC\subset \A^3$ is a 
{\em $C^p$ invariant cylinder with boundary} for $X$ if $\jC$ is a submanifold  of 
$\A^3$,  $C^p$--diffeomorphic to $\T^2\times [0,1]$, such that $X$ is everywhere tangent 
to $\jC$ and is moreover tangent to $\partial \jC$ at each point of $\partial \jC$.  

\item Given an invariant cylinder with boundary $\jC$, we say that it is {\em normally hyperbolic}
when there exist a neighborhood $N$ of $\jC$ and a complete vector field $X_\circ$ on $\A^3$ such 
that  $X\equiv X_\circ$ in $N$ and such that $X_\circ$ admits a normally hyperbolic invariant submanifold 
$\jC_\circ$,  diffeomorphic to  $\T^2\times\R$, which contains $\jC$.
\end{itemize}
\end{definition}

Note first that $\jC$ is invariant under the  flow, thanks to the tangency hypothesis on the boundary. 
In particular, both connected components of $\partial \jC$ are invariant $2$-dimensional tori.
In the following, when the context is clear, normally hyperbolic cylinders with boundary  will be called  
{\em compact invariant cylinders} for short.

\vskip1mm

The main interest of the previous definition is that it is possible to properly define the stable and 
unstable manifolds of compact invariant cylinders.
Indeed, one checks that the stable manifold $W^{ss}(x)$ of a point $x\in \jC$ is well-defined and 
{\em independent of the choice of} $(X_\circ,\jC_\circ)$ (recall that $W^{ss}(x)$ is the set
of all initial conditions $y$ such that ${\rm dist}(\Phi^{tX}(x),\Phi^{tX}(y))$ tends to $0$ at an 
exponential rate $e^{-ct}$, where $c$ dominates the contraction exponent on $\jC_\circ$).
The stable  manifold of  $\jC$ is then well-defined as the union of the stable manifolds $W^{ss}(x)$ 
for $x\in \jC$, which turns out to have the same regularity as $\jC$. The same remark obviously also 
holds for the unstable manifolds. 


\paraga  In addition to our previous invariant cylinders, it will be necessary to introduce slightly more 
complicated objects which we call {\em singular cylinders}. Recall that $\cA_\bu$ is the standard
singular annulus defined in the previous section.

\begin{definition}  Let $X$ be a vector field  on $\A^3$. 
\begin{itemize}
\item We say that $\jC_\bu\subset \A^3$ is an
{\em invariant singular cylinder} for $X$ if $\jC_\bu$ is a $C^1$ submanifold with boundary of 
$\A^3$,   diffeomorphic to $\T\times \cA_\bu$, such that $X$ is everywhere tangent to $\jC_\bu$ 
(and is moreover tangent to $\partial \jC_\bu$ at each point of $\partial \jC_\bu$).  

\item Given an invariant singular cylinder $\jC_\bu$, we say that it is {\em normally hyperbolic}
when there exist a neighborhood $N$ of $\jC_\bu$ and a complete vector field $X_\circ$ on $\A^3$ 
such that  $X\equiv X_\circ$ on $N$ and which admits a normally hyperbolic invariant submanifold 
$\jC_\circ$, diffeomorphic to  $\T^2\times\R$, which contains  $\jC_\bu$. 
\end{itemize}
\end{definition}

As above, we simply say {\em compact singular cylinders} instead of normally hyperbolic compact 
invariant singular cylinders. Again, the  stable and unstable manifolds  of a point 
$x\in \jC_\bu$ are well-defined  and independent of the choice of  $(X_\circ,\jC_\circ)$,  and this is also 
the case for the stable and unstable manifolds of  $\jC_\bu$.


\paraga Let $H$ be a Hamiltonian on $\A^3$ and let $\e$ be a regular value of $H$.

\begin{definition}
A {\em chain of cylinders} for $H$ at energy $\e$ is a finite family $(\jC_i)_{1\leq i\leq i^*}$ of 
compact invariant cylinders {\em or singular cylinders}, contained in $H\inv(\e)$,
such that $W^u(\jC_i)$ intersects $W^s(\jC_{i+1})$ for $1\leq i\leq i^*-1$. 
\end{definition}

Note in particular that, for the sake of simplicity, we do not make any distinction between ``regular'' cylinders
and singular cylinders in a chain. Note also that the definition here slightly differs from that of
chains of annuli above. In the following we will have to add suitable transversality conditions for the 
various homoclinic and  heteroclinic  intersections in a chain of cylinders, which could be stated in 
a general context but will  be easier to make explicit  in the case of our Hamiltonians $H_n$, 
this will be done in  Section~\ref{sec:pert}.


\subsection{Cylinders for $\H_n$}\label{sec:cyl}

\paraga We now go back to the truncated Hamiltonian $\H_n$ defined in (\ref{eq:truncham}). 
Let $k=(k_2,k_3)\in\Z^2$ be a primitive integer vector (that is, $k_2$ and $k_3$ are coprime) 
and let $\S_{k}$ be the half great circle
of the unit sphere $\S$ formed by the actions $r=(r_1,\ov r)=(r_1,r_2,r_3)$ such that
$$
\ov r\cdot  k=0,\qquad (-r_3,r_2) \cdot  k \geq 0,\qquad r\in \S.
$$
We denote by
$\dd$ the Hausdorff distance for compact subsets of $\R^3$.
The main result of this section is the following.


\begin{prop}\label{prop:cyl}  Let $U\in \jU$ and set, for $(\th,r)\in\A^3$
$$
\H_n(\th,r)=\pdemi\norm{r}^2+\tfrac{1}{n} U(\th_2,\th_3).
$$
Fix $ k$ as above and fix $\de >0$. Then:
\begin{itemize}
\item  there is $n_0(k)>0$ such that for $n\geq n_0(k)$, there are 
regular cylinders $\jC_{-m},\ldots,\jC_{-1}$, $\jC_0$, $\jC_1,\ldots,\jC_m$, where the 
integer $m$ depends on $k$, which satisfy
\beq\label{eq:prox}
\dd\big(\cup_j\Pi(\jC_j),\S_k\big)<\de,
\eeq
and such that both ordered families $\jC_{-m},\ldots,\jC_{-1}$, $\jC_0$, $\jC_1,\ldots,\jC_m$
and $\jC_{m},\ldots,\jC_{1}$, $\jC_0$, $\jC_{-1},\ldots,\jC_{-m}$ are chains;

\item there exist two singular cylinders $\jC_\bu^-$ and $\jC_\bu^+$, {\em independent of $k$}, 
such that the extremal cylinders $\jC_{-m}$ and $\jC_m$ admit transverse heteroclinic connections with
$\jC_\bu^-$ and $\jC_\bu^+$ respectively;

\item each cylinder $\jC_j$ admits a foliation by isotropic $\T^2$ tori, such that the union of the subfamily
of dynamically minimal tori is a dense subset of $\jC_j$, and each singular cylinder $\jC^\pm_\bu$
admits a foliation by isotropic $\T^2$ tori on an open and dense subset, such that the union of the subfamily
of dynamically minimal tori is a dense subset of $\jC^\pm_\bu$. 
\end{itemize}
\end{prop}


We will moreover prove that the cylinders $\jC_j$ and $\jC_{-j}$ are exchanged with one another 
by a natural symmetry.


\begin{proof} We can assume without loss of generality that $\ha e=\Max U=0$.
We first perform a standard rescaling to get rid of the parameter $n$, namely, setting for $a>0$:
\begin{equation}\label{eq:rescaling1}
\sig_a(\th,r)=(\th,a\,r),
\end{equation}
one immediately checks the conjugacy relation
\begin{equation}\label{eq:rescaling2}
\Phi_{nt}^{\H_n}=\sig_{\sqrt n}\inv\circ \Phi_t^\H\circ\sig_{\sqrt n},
\end{equation}
where $\H:=\H_1$, while $\sig_{\sqrt n}$ sends the energy level $\H_n\inv(\pdemi)$ onto the level 
$\H\inv(\tfrac{n}{2})$. We can therefore examine the behavior of the system $\H$ at high energy $\e$
and get our results by inverse rescaling. We will fix two coprime integers $(k_2,k_3)$ and 
concentrate on the neighborhood of the half great circle $\sqrt{2\e}\,\S_k$ on the sphere of radius 
$\sqrt{2\e}$.


\vskip2mm $\bu$ We will apply Theorem~\ref{thm:genprop} to $c\sim (-k_3,k_2)$.
Reversing the order of the intervals in this theorem, for compatibility reasons with the statement 
of Proposition~\ref{prop:cyl},  there is an ordered family  $I_m,\ldots, I_0$, with  $I_m=\,]0,e_m]$ 
and $I_0=[e_P,+\infty[$  such that  the system $C_U$ admits a chain of annuli 
$\sA_m,\ldots,\sA_0$ realizing $c$ and defined over $I_m,\ldots, I_0$. For $0\leq j\leq m$,
we denote by $\phi_j:\T\times I_j\to \A^2$ the embedding of $\sA_j$ satisfying~(\ref{eq:embedd}),
and for $e\in I_j$, we denote by $\Ga_j(e)=\phi_j(\T\times\{e\})$ the periodic orbit at 
energy $e$ in $\sA_j$.


\vskip2mm $\bu$ Let us fix an energy $\e>e_P$. 
The level $\H\inv(\e)$ contains the union
$$
\bigcup_{0\leq e_1\leq \e} \big\{\th_1\in\T,\ \pdemi r_1^2=e_1\big\}\times C_U\inv(\e-e_1).
$$
which will serve as a guide to construct embeddings for our cylinders.


\vskip2mm $\bu$ Consider first an annulus $\sA_j$ with $1\leq j\leq m$ and set
$I_j=[a_j,b_j]$ if $1\leq j\leq m-1$ and $I_m=\,]0,b_m]$, so that $I_j$ is contained in 
$]0,e_P]$ (recall that $I_0=[e_P,+\infty[$). We 
introduce the map
\beq\label{eq:embedF+}
\begin{array}{ll}
&F_j^+: \T^2\times I_j\longrightarrow\H\inv(\e)\subset\A\times\A^2\\
&F_j^+(\ph_1,\ph_2,e)=\Big(\big(\ph_1,\sqrt{2(\e-e)}\big),\phi_j(\ph_2,e)\Big).
\end{array}
\eeq
One immediately checks that $F^+_j$ is an embedding. Let $\jC_j\subset \H\inv(\e)$ be its image. Then
$\jC_j$ is a cylinder (diffeomorphic to the product of $\T^2$ with an interval),  which is a circle bundle over
the annulus $A_j$. So $\jC_j$ admits a regular foliation 
whose leaves are the tori
$$
\jT_{e}=F_j^+(\T^2\times\{e\}).
$$
The torus $\jT_{e}$ is the direct product of the circle $\T\times\{\sqrt{2(\e-e)}\}$ in the first
factor of the product $\A\times\A^2$ with the hyperbolic periodic orbit 
$\Ga_j(e)$ in the second factor. For each point $z=(z_1,z_2)\in\A\times\A^2$ in $\jC_j$, there 
is a single hyperbolic orbit in the annulus $\sA_j$ which contains $z_2$. This yields a  decomposition of  
the tangent  space $T_z\H\inv(\e)$ as a sum $T_z\jC_j\oplus E^+(z)\oplus E^-(z)$, 
where $E^\pm(z)$ are the stable and unstable directions of that hyperbolic orbit at the point $z_2$. 
All these considerations also make sense for any small enough hyperbolic continuation of $\sA_j$, which
immediately proves that $\jC_j$ is a compact invariant cylinder in the sense of Definition \ref{def:invcyl}.


\vskip2mm $\bu$ One gets a parallel construction using the embedding
\beq\label{eq:embedF-}
\begin{array}{ll}
&F_j^-: \T^2\times I_j\longrightarrow\H\inv(\e)\subset\A\times\A^2\\
&F_j^-(\ph_1,\ph_2,e)=\Big(\big(\ph_1,-\sqrt{2(\e-e)}\big),\phi_j(\ph_2,e)\Big).
\end{array}
\eeq
whose image will be denoted by $\jC_{-j}$ and is a compact invariant cylinder as well. 
Moreover, $\jC_j$ and $\jC_{-j}$ are obviously symmetric under $r_1\mapsto -r_1$.


\vskip2mm $\bu$ To exhibit the singular cylinders, one fixes an embedding $\phi_\bu:\cA_\bu\to\A^2$
satisfying~(\ref{eq:fibu}) (with suitable parameters),
whose image is the singular annulus $\sA_\bu$ of the system $C_U$ depicted in Theorem~\ref{thm:genprop}. 
This enables one to introduce two maps
\beq\label{eq:embedFbu}
\begin{array}{ll}
&F_\bu^\pm:\T\times\cA_\bu\longrightarrow \H\inv(\e)\\
&F_\bu^\pm\big(\ph_1,(\ph_2,r_2)\big)=
\Big(\Big(\ph_1,\pm \sqrt{2(\e-C_U\big(\phi_\bu\big(\ph_2,r_2)\big)}\Big),\phi_\bu(\ph_2,r_2)\Big).
\end{array}
\eeq
Again, one easily checks that these are embeddings. Their images $\jC_\bu^\pm$ are singular cylinders for $\H$,
which are symmetric under $r_1\mapsto-r_1$.


\vskip2mm $\bu$ The case of $\sA_0$ is slightly different, since the energy $e$ can no longer be used as a global
parameter on it. Rather,  we introduce the interval 
$J_0=\big[-\sqrt{2(\e-e_P)},\sqrt{2(\e-e_P)}\big]$ and the map
\beq\label{eq:embedF0}
\begin{array}{ll}
&F_0 :\T^2\times J_0\longrightarrow\H\inv(\e)\subset\A\times\A^2\\
&F_0(\ph_1,\ph_2,r_1)=\Big((\ph_1,r_1),\phi_0\big(\ph_2,\e-\pdemi r_1^2\big)\Big).
\end{array}
\eeq
One easily checks that $F_0$ is again an embedding and that its image $\jC_0$ is a compact invariant 
cylinder (note that now  $\jC_0$ is a two sheeted ramified circle bundle over the corresponding part of $\sA_0$). 
This completes the construction of the family $\jC_{-m},\ldots,\jC_m$ and $\jC_\bu^\pm$.

%


\vskip2mm $\bu$ Fix now an integer $j\in \{0,\ldots,m-1\}$ and fix $e\in[\Max I_{j+1},\Min I_j]$. By definition of a chain
of annuli, there exists a heteroclinic connection 
$$
\Om^{j+1}_{j}\subset W^u(\Ga_j(e))\cap W^s(\Ga_{j+1}(e))
$$
between extremal periodic orbits of  $\sA_{j}$ and $\sA_{j+1}$, which gives rise to an 
{\em annulus} of heteroclinic orbits (diffeomorphic to $\A$) between $\jC_j$ and $\jC_{j+1}$, namely
$$
\big(\T\times\{\sqrt{2(\e-e)}\}\big)\times \Om_j^{j+1}.
$$
 Again, a parallel construction using now the heteroclinic connection
$$
\Om^j_{j+1}\subset W^u(\Ga_{j+1}(e)) \cap W^s(\Ga_j(e))
$$ 
proves the existence of an annulus of heteroclinic orbits 
between $\jC_{-(j+1)}$ and 
$\jC_{-j}$. We indeed get more heteroclinic connections by considering the whole overlapping 
energy interval $[\Max I_{j+1},\Min I_j]$.
This proves that the family $\jC_{-m},\ldots,\jC_m$ is a chain of cylinders. 

\vskip1mm
The proof for the
opposite ordering $\jC_{m},\ldots,\jC_{-m}$ is similar.
Finally, the existence of  annuli of heteroclinic connections between $\jC_{\pm m}$ and 
$\jC_\bu^\pm$ follows from exactly the same considerations as above.


\vskip2mm $\bu$ The cylinders $\jC_j$, $1\leq j\leq m$, are foliated by the invariant tori $\jT_{e}$
for $e\in I_j$. Let us prove that, when $\e$ is large enough, they are dynamically minimal for  
$e$ in a dense subset of  $I_j$. Fix the integer $j$ and let   $\om_j:I_j\to \R$ be the the frequency map 
of the annulus $\sA_j$.
The frequency on the first factor (according to the component form of $F^+_j$) of $\jT_e$, that is,
the circle $\T\times\{\sqrt{\e-e}\}$ is $\sqrt{\e-e}$, so 
that the frequency vector of $\jT_e$ is 
$$
\Om(e)=\big(\sqrt{\e-e},\om_j(e)\big).
$$
Now, by Definition~\ref{def:annuli} and by compactness
of $I_j$, 
the frequency map  $\om_j:I_j\to \R$ of the annulus $\sA_j$ satisfies 
$$
\om'_j(e)\geq \mu >0
$$ 
for $e\in I_j$.  
This proves that the frequency curve $\Om(e)\subset\R^2$ of $\jC_j$  is transverse to each 
vector line in $\R^2$, so that the ratio $\Om_2/\Om_1$ is irrational for $e$ in a dense subset of $I_j$. 
The corresponding torus $\jT_{e}$ is therefore dynamically minimal. 
Similar arguments show the same property for the cylinders $\jC_j$, $-m\leq j\leq -1$, as well as for
the singular cylinders $\jC_\bu^\pm$.


\vskip2mm $\bu$ It remains to examine the torsion properties of $\jC_0$. Observe that
up to a standard linear change of variables (see \cite{Mar1}), one can assume that $c=(1,0)$. In the new 
variables, that we still denote by $(\th,r)$, the kinetic part of the Hamiltonian $\H$ takes the form 
$$
T(r)=\pdemi\big(r_1^2+Q(r_2,r_3)\big),
$$
and for $\e$ large enough $X^T$ is the dominant term of $X^\H$ since $X^U$ is bounded. 
Moreover, when the energy $e$ of the classical system is large enough, 
$$
C_u(\th_2,\th_3,r_2,r_3)\sim \pdemi Q(r_2,r_3).
$$
Now, by the asymptotic property of the projection of the Poincar\'e annulus $\sA_0$ (that is, the annulus $\sA_m$ in 
Theorem~\ref{thm:genprop}), the frequency map of $\sA_0$ is a $o(1)$ $C^2$ pertubation of the map
$$
\om: e\sim\pdemi Q(r_2,r_3)\longmapsto \d_{r_2} Q(r_2,r_3).
$$
As a consequence the frequency vector of the corresponding tori $\pm\jT(e)$ on $\jC_0$ is a small $C^2$
perturbation of 
$$
\Om(e)=(\sqrt{\e-e},\om(e)).
$$
Since $\om'(e)\to+\infty$ when $e\to\infty$,
this proves now the density of the torsion property for the two connected components of $\jC_0$ 
fibered over the subannulus of $\sA_0$ defined over $[e^*,\e[$, where $e^*$ is  large 
enough  (we obviously always require $\e>e^*$ for being consistent). 

The case of the remaining components of $\jC_0$ 
associated with the subannulus defined over $[e_P,e^*[$ is analogous to that of the cylinders $\jC_j$
since one can choose on this part a parametrization similar to those given by the embeddings $F_j^\pm$.

As a consequence, the union of the subset of dynamically minimal tori of  $\jC_0$ is dense
in $\jC_0$.


\vskip2mm $\bu$ It only remains to prove (\ref{eq:prox}), but this is an immediate consequence
of Theorem \ref{thm:genprop}, taking into account the reparametrization (\ref{eq:rescaling2}).
This concludes the proof.
\end{proof}

\begin{figure}[h]
\begin{center}
\begin{pspicture}(0cm,3cm)
\psset{unit=.7cm}
\psline[linewidth=0.1mm]{->}(0,-4)(0,4)
\psline[linewidth=0.1mm]{->}(-4,0)(4,0)
\psline[linewidth=0.1mm]{->}(3,3)(-3,-3)
\rput(-1.6,4){$\Pi(\jC_0)$}
\psline[linewidth=0.05mm](-.8,2)(-1.4,3.5)
\rput(2.6,4){$\Pi(\jC_1)$}
\psline[linewidth=0.05mm](1.4,2.2)(2.4,3.5)
\rput(-5,-1){$\Pi(\jC_{-1})$}
\psline[linewidth=0.05mm](-4,-1)(-1.7,-.55)
\pscircle[linewidth=0.1mm](0,0){3.04}
\rput{15}{
\psscalebox{2 -1}{\psarc[linewidth=0.1mm](0,0){1.5}{0}{180}}
}
\rput{15}{
\psscalebox{2 -1}{\psarc[linestyle=dashed,linewidth=0.1mm](0,0){1.5}{180}{360}}
}
\rput{-15}{
\psscalebox{1 2}{\psarc[linewidth=0.1mm](0,0){1.5}{40}{220}}
}
\rput{-15}{
\psscalebox{1 2}{\psarc[linewidth=0.5mm](0,0){1.55}{70}{190}}
}
\rput{-15}{
\psscalebox{1 2}{\psarc[linewidth=0.5mm](0,0){1.48}{50}{70}}
}
\rput{-15}{
\psscalebox{1 2}{\psarc[linewidth=0.5mm](0,0){1.48}{190}{210}}
}
\rput{-15}{
\psscalebox{1 2}{\psarc[linewidth=0.5mm](0,0){1.56}{45}{50}}
}
\rput{-15}{
\psscalebox{1 2}{\psarc[linewidth=0.5mm](0,0){1.56}{210}{215}}
}
\rput{-15}{
\psscalebox{1 2}{\psarc[linewidth=0.5mm](0,0){1.46}{40}{45}}
}
\rput{-15}{
\psscalebox{1 2}{\psarc[linewidth=0.5mm](0,0){1.46}{215}{220}}
}
\end{pspicture}
\vskip3cm
\caption{Projected cylinders}
\end{center}
\end{figure}
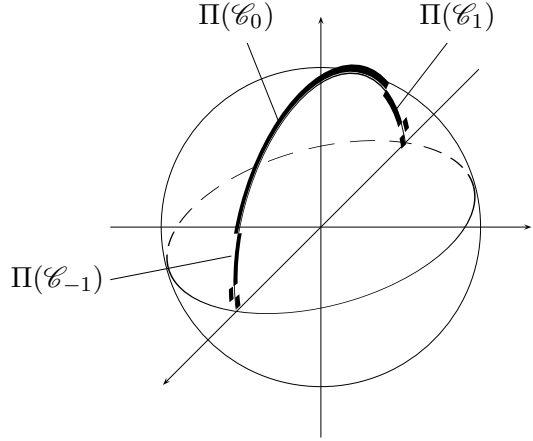


\begin{definition} \label{def:associated}
We will say that the cylinders $\jC_{\pm j}$ exhibited in Proposition \ref{prop:cyl} are associated with the annulus 
$\sA_j$, for $1\leq j\leq m$, and that $\jC_\bu^\pm$ is associated with $\sA_\bu$. In the same way, we say that
the embeddings $F^\pm_j$ and $F^\pm_\bu$ are associated with the cylinders $\jC_{\pm j}$ and $\jC_\bu^\pm$.
Note that each (singular) embedding $F^+_\bu$ or $F^-_\bu$ gives rise to three regular embeddings associated
with the embeddings $\phi^>_\bu$, $\phi^<_\bu$ and $\phi^0_\bu$ of Remark~\ref{rem:emb}.

We will in general denote by $\T^2\times \I$ the domain of an embedding associated with a regular cylinder, without 
specifying the particular form of the parametrization, and we will get rid of the various indices for the embeddings.
\end{definition}

In the following we will apply the previous proposition to an increasing family
of simple resonances $\cup_{1\leq \ell\leq n} \S_{k_\ell}$, and we need to exhibit a single chain of 
cylinders for $\H_n$ whose projection follows each semi-circle in this family.  To this aim,
we order the subset $\ha \Z^2\subset\Z^2$ formed by the primitive vectors $k$,  
in such a way that the resulting sequence $(k_\ell)_{\ell\geq1}$ satisfies 
$\norm{k_\ell}\leq \norm{k_{\ell+1}}$. For each $k\sim c\in\ha \Z^2$, we denote by 
$\cyl_{k}(\H_n)$ the set of cylinders associated with the annuli $\sA_0(c),\ldots,\sA_m(c)$ of 
Theorem~\ref{thm:genprop}, together with the singular cylinders $\jC_\bu^\pm$, so that
$\#\cyl_{k}(\H_n)=2m+3$. Finally we set
$$
\cyl(\H_n)=\bigcup_{1\leq\ell\leq n}\cyl_{k_\ell}(\H_n).
$$
Recall that by Proposition~\ref{prop:cyl}, 
for each $k\in\ha\Z^2$, the cylinders in $\cyl_{k}(\H_n)$ form  {\em two} chains depending on the way they are ordered, namely:
$$
\chain^+_{k}(\H_n):\quad \jC_\bu^-,\jC_{-m},\ldots,\jC_0,\ldots,\jC_m,\jC_\bu^+,
$$
and
$$
\chain^-_{k}(\H_n):\quad \jC_\bu^+,\jC_{m},\ldots,\jC_0,\ldots,\jC_{-m},\jC_\bu^-.
$$

\begin{definition} 
Let $n\geq1$ be fixed. We denote by $\chain(\H_n)$ the chain formed by the concatenation of the 
chains $\chain^{(-1)^\ell}_{k_\ell}(\H_n)$, $1\leq \ell\leq n$.
\end{definition}

Finally, we denote by $\tori(\H_n)$ the set of all $\T^2$--tori of the form $\jT_{e}$ contained in
the cylinders and singular cylinders of $\cyl(\H_n)$.
The rest of the paper is devoted to the construction of a perturbation which will create shadowing orbits
along $\chain(\H_n)$, passing close to a $\de$-dense family of dynamically minimal tori in 
$\tori(\H_n)$.


\section{Construction of the perturbation}\label{sec:pert}
In this section, we will construct the perturbation $f_n\in C^{\ka-1}(\A^3)$ such that the system  
$H_n=\H_n+f_n$ will admit the same family of cylinders as $\H_n$, with additional transversality properties 
for their invariant manifolds (the so-called splitting of separatrices). 


\subsection{The transversality conditions}
We first set out some definitions for the splitting of separatrices. In this section $f_n$ denotes a function in 
$C^{\ka-1}(\A^3)$ whose support is contained in the complement of the union of the cylinders of $\cyl(\H_n)$.
As a consequence,  each $\jC\in \cyl(\H_n)$ is also invariant under the flow generated by $H_n=\H_n+f_n$ 
and contained in  $H_n\inv(\pdemi)$.  We can therefore set
$$
\cyl(H_n):=\cyl(\H_n),\qquad \chain(H_n):=\chain(\H_n),\qquad \tori(H_n):=\tori(\H_n).
$$
Given a  point $x$ in a cylinder of $\cyl(H_n)$, note that its stable and unstable manifolds $W^{s,u}(x)$ are 
well-defined, and  that this is also the case for the stable and unstable manifolds of any $\jT\in\tori(H_n)$, 
that we denote by $W^{s,u}(\jT)$.
Let us now introduce our conditions.

\begin{definition}\label{def:weakhet}
Let $\jT\in\jC$ and $\jT'\in\jC'$ be two elements of ${\rm Tori}(H_n)$ (recall that $\jC$ and $\jC'$ can be 
singular  cylinders). We say that the pair $(\jT,\jT')$  satisfies {\em the weak splitting condition} if there exists  $a\in\jT$ 
whose  unstable  manifold $W^{uu}(a)$ intersects $W^s(\jC')$  transversely in $H_n\inv(\pdemi)$ at some 
point of  $W^s(\jT')$.
\end{definition}

Note that our definition differs from the usual ones in which one requires the Lagrangian invariant manifolds
of $\jT$ and $\jT'$ to intersect transversely. This latter splitting condition is obviously stronger than ours.

\begin{definition}\label{def:cylT}
Let $\jC$ be a regular cylinder in $\cyl(H_n)$ and let $F$ be the associated embedding, $F:\T^2\times \I\to H_n\inv(\pdemi)$ 
and for $s\in \I$, set $\jT(s)=F(\T^2\times\{s\})$.
We say that $\jC$ satisfies {\em Condition (T)} if there exists $\rho>0$ such that for each pair 
$(s,s')\in \I^2$ with $\abs{s-s'}<\rho$, the  pair $(\jT(s),\jT(s'))$ satisfies the weak splitting condition.
\end{definition}

Let now $\jC_\bu=\jC_\bu^\pm$ be a singular cylinder, together with its subcylinders $\jC_\bu^*$,
where $*$ stands for $>,<$ or $0$, and the associated embeddings $F_\bu^*$ defined over $\I^*$. 
Note that, by the construction of the previous section, the natural parameter in $\I^*$ is the energy
$e$ of the classical system $C_U$.
We set $\jT^*(e)=F_\bu^*(\T^2\times\{e\})$.

\begin{definition} \label{def:cylsingT} 
We say that $\jC_\bu$ satisfies {\em Condition~(T)} when each subcylinder $\jC_\bu^*$ satisfies Condition~(T) and when moreover
there exists $\rho>0$ such that if $\abs{e-e'}<\rho$, each pair $(\jT^<(e),\jT^0(e'))$ and $(\jT^>(e),\jT^0(e'))$
satisfies the weak splitting condition.
\end{definition}

In other words, we want two close enough subtori in $\jC_\bu$ to satisfy the weak splitting condition, 
when one of them is located in $\jC_\bu^0$ and the other one is in $\jC_\bu^>$ or $\jC_\bu^<$.

\vskip2mm

As for  chains now, we have to add a transversality condition for tori contained in distinct cylinders.

\begin{definition} 
We say that a  chain of cylinders $(\jC_k)_{1\leq k\leq k^*}$ satisfies 
{\em Condition~(S)} when each cylinder  $\jC_k$ satisfies Condition (T) and when moreover,
for  $1\leq k\leq k^*-1$, there are open subsets $O_k\subset\jC_k$ and $O_{k+1}\subset\jC_{k+1}$,
union of elements of $\tori(H_n)$, such that for each $\jT\subset O_k$ and $\jT'\in O_{k+1}$,
the pair $(\jT,\jT')$ satisfies the weak splitting condition.
\end{definition}

Note that Condition (S) is obviously open in the following sense.

\begin{lemme}
Assume that $\chain(H_n)$  satisfies Condition~(S). Given a small enough function $f$ in the 
$C^2$ topology, with support contained in the complement of the union of the cylinders of $\cyl(H_n)$, 
then $\chain(H_n+f):=\chain(H_n)$ is a chain at energy $\pdemi$ for  $H_n+f$ and satisfies 
Condition (S).
\end{lemme}

Our purpose in this section is to prove the following result.

\begin{prop}\label{prop:condT} Fix $\ka\geq2$. Then for each $n\geq1$,  there exists a function 
$f_n\in C^{\ka-1}(\A^3)$, whose support is contained in the complement of the union of the cylinders of 
$\cyl(\H_n)$, and which satisfies $\norm{f_n}_{C^{\ka-1}(\A^3)}\leq \tfrac{1}{n}$, such that
$\chain(\H_n+f_n)$ satisfies Condition~(S).
\end{prop}

The rest of the section is devoted to the proof which requires two steps. We will first exhibit a 
perturbation which creates the heteroclinic connections for the pairs of tori contained in the same cylinder
and we will then use the previous openness property to add a second perturbation  adapted to the
heteroclinic conditions for extremal tori of the chain. 


\subsection{Flow boxes near homoclinic intersections of cylinders}\label{sec:flowbox}
In order to properly define the various pertubations, we first need to construct ``flow boxes'' centered on 
suitable parts of 
the homoclinic manifolds of the  cylinders of $\H_n$, and located ``far from these cylinders''.  Given $U\in\jU$,
we fix an annulus $\sA$ of the system $C_U$ defined over $I$. We denote by $\Ga(e)$ the periodic orbit 
$\sA\cap C_U\inv(e)$ and we fix an open subinterval $I^*\subset I$ 
over which $W^{s,u}(\Ga(e))$ intersect transversely along a homoclinic orbit $\Om(e)$ 
(see Definition \ref{def:annuli}). Therefore, there exists a  $3$--dimensional section $\Sig$ in $\A^2$,
transverse to $X^{C_U}$,  such that the union $\cup_{e\in I^*}\Om(e)$ intersects $\Sig$ along a $C^{\ka-1}$ curve
$\sig$.

\paraga  Since $\Sig$ is transverse to $X^{C_U}$, for $e\in I^*$ the intersection $\Sig\cap C_U\inv(e)$ is
symplectic. Reducing $\Sig$ if necessary, one easily proves the existence of a ball $B=[-\de,\de]^2\subset\R^2$ centered
at $0$ and a $C^{\ka-1}$  diffeomorphism  $\chi_0:I^*\times B\to \Sig$, such that 
\begin{itemize}
\item $C_U\circ\chi_0(e,s,u)=e$;
\item the connected component of $W^u(\sA)\cap \Sig$ containing $\sig$ admits the equation $s=0$; 
\item the connected component of $W^s(\sA)\cap \Sig$ containing $\sig$ admits the equation $u=0$;
\item for each $e\in I^*$, $\chi_0(e,\cdot)$ is symplectic for the usual structure on $B$
and the induced structure on $\Sig\cap C_U\inv(e)$.
\end{itemize}

\paraga For $\tau_0>0$ small enough, the Hamiltonian flow $\Phi^{C_U} : \,]-\tau_0,\tau_0[\,\times\, \Sig \to \A^2$
is a  diffeomorphism onto its image $\jO$. One easily checks that one can choose the previous coordinates $(u,s)$
in such a way that the map
$$\chi: \,]-\tau_0,\tau_0[\,\times\, I^*\times B \longrightarrow \jO$$
$$
\chi(\tau,e,s,u)=\Phi_{\tau}^{C_U}\big(\chi_0(e,s,u)\big).
$$
is a $C^{\ka-1}$ symplectic diffeomorphism.
By construction, the Hamiltonian
$C_U$ takes the simple expression 
$$
C_U\circ\chi(\tau,e,s,u)=e.
$$
This in turn yields a $C^{\ka-1}$ symplectic diffeomorphism 
$\ha\chi:\jD\longrightarrow \A\times \jO\subset\A^3$, where 
$\jD$ is the subset of all $(\tau,\e,s,u,\th_1,r_1)\in\,]-\tau_0,\tau_0[\,\times\, \R\times B \times\A$ such that 
$\e-\pdemi r_1^2\in I^*$ (note that now $\e$ stands for the {\em total} energy of the system), defined by
$$
\ha\chi(\tau,\e,s,u,\th_1,r_1)
=\Phi^{\H}_{\tau}\Big((\th_1,r_1),\chi\big(0,\e-\pdemi r_1^2,s,u\big)\Big)
=\Big((\th_1+\tau r_1,r_1),\chi\big(\tau,\e-\pdemi r_1^2,s,u\big)\Big),
$$
which clearly satisfies
$$
\H\circ \ha\chi(\tau,\e,s,u,\th_1,r_1)=\e.
$$
We will set
$$
\ha\Sig=\ha\chi\big(\{(0,\pdemi)\}\times\Sig\big)\subset\H\inv(\e)
$$
so that $\jD$ is a neighborhood of $\ha\Sig$ in $\A^3$.

\paraga The effect of the rescaling (\ref{eq:rescaling2}) is immediately computed in the previous straightening 
coordinates. We set
\beq\label{eq:rescalegeom}
\ha\chi_n=\sig_{\sqrt n}\inv\circ\ha\chi,\qquad \ha\Sig_n=\sig_{\sqrt n}\inv(\ha\Sig).
\eeq
Since for $t$ small enough
\beq\label{eq:rescaletime}
{\ha\chi}\inv\circ\Phi^{\H}_{t}\circ\ha\chi(\tau,\e,s,u,\th_1,r_1)=(\tau+t,\e,s,u,\th_1,r_1)
\eeq
one immediately gets
\beq\label{eq:rescstraight}
\ha\chi_n\inv\circ\Phi^{\H_n}_t\circ\ha\chi_n(\tau,\e,s,u,\th_1,r_1)=(\tau+\tfrac{1}{n}\,t,\e,s,u,\th_1,r_1),
\eeq
so that
$$
\H_n\circ\ha\chi_n(\tau,\e,s,u,\th_1,r_1)=\frac{\e}{n}.
$$


\subsection{Perturbation and Condition (T) for cylinders}
In this section we fix $n\geq 1$ and we work with the Hamiltonian $\H_n$.
We now construct a first perturbation $f_n^{(1)}$ which produces heteroclinic connections
between nearby elements of $\tori(\H_n+f_n^{(1)})$ contained in the same cylinder and yields 
Condition (T) for each cylinder.

\vskip1mm

To begin with, let us consider a regular cylinder $\jC\in\cyl(\H_n)$, associated with some annulus 
$\sA$ of $C_U$ defined over an interval $I$, and let $I^*$ be a subinterval of $I$ as in the previous section. 
Let $F:\T^2\times I\to \H_n\inv(\pdemi)$ be the associated embedding of $\jC$ (where we assume that
the natural parameter is the energy, the case of the annulus $\sA_0$ is easily treated using similar arguments).  
As usual, we set $\jT_{e}=F(\T^2\times\{e\})$.

\vskip2mm

We will use the previous flow box coordinates to construct the perturbation. 
With the assumptions and notation of the previous section, we define a function $f$ by
\beq\label{eq:fonctionf}
\til f\circ\ha\chi_n(\tau,\e,s,u,\th_1,r_1)=\mu\eta_\tau(\tau)\eta_\th(\th_1),
\eeq
where $\mu>0$ is a small enough constant and 
$
\eta_\tau
$
is a (nonzero) $C^\infty$ bump function whose support is located in $[-\tau_0,0]$ and which takes its values in $[0,1]$,
while $\eta_\th\in C^{\infty}(\T,[0,1])$ is a smooth function with support in $[-1/4,1/4]$ whose
derivative vanishes only at $0$ in  the interval $]-1/4,1/4[$. In particular the support of $f$ is contained in
the domain $\ha\chi_n(\jD)$ for coherence.

\vskip2mm

\begin{lemme}\label{lem:weaktran1} Fix a regular cylinder $\jC\in\cyl(\H_n)$, let $f$ be as in~(\ref{eq:fonctionf}) and
set $H_n=\H_n+f$, with $f=\til f\circ \ha\chi_n\inv$.
Then
there exist $\mu>0$, 
$\eta_\tau$ in $C^{\infty}(\R,[0,1])$, $\tau_\th$ in $C^\infty(\T,[0,1])$ and $\rho>0$ such that the pair of invariant tori $(\jT_{e},\jT_{e'})$ 
for $H_n$ satisfies the weak splitting condition for $e,e'\in I^*$ with $\abs{e-e'}<\rho$, and such that moreover  
$$
\norm{f}_{C^{\ka-1}(\A^3)}\leq \nu.
$$
\end{lemme}

\begin{proof} We have fixed the energy $\e=\pdemi$. 
We first choose
$\mu>0$ small enough so that the composition 
$
f=\til f\circ\ha\chi_n\inv
$
(which is of class $C^{\ka-1}$ since $\ha\chi_n$ is) satisfies $\norm{f}_{C^{\ka-1}}\leq\nu$, which is obviously possible.
The coordinates $(\tau,s,u,\th_1,r_1)$
form, via the symplectic diffeomorphism $\ha\chi_n$, a chart in the neighborhood of $\ha\Sig_n$ in $H_n\inv(\pdemi)$. 
In this chart,  the vector field generated by 
$H_n\circ \ha\chi_n$  reads:
\beq\label{eq:vectorbox}
\left\vert
\begin{array}{lllll}
\dot \tau= \frac{1}{n}\\
\dot s=0\\
\dot u=0\\
\dot \th_1=0\\
\dot r_1=\mu\eta_\tau(\tau)\eta_\th'(\th_1)\\
\end{array}
\right.
\eeq
Therefore, the
variation of the variable $r_1$ when passing through the support of the function $f$ is easily computed:
$$
\De r_1=\mu\norm{\eta_\tau}_1\eta_\th'(\th_1),
$$
where $\norm{\eta_\tau}_1\neq0$ is the $L^1$ norm.
The conclusion for the existence of transverse heteroclinic intersections now easily follows: in the chart
$(\tau,s,u,\th_1,r_1)$ the intersection $\ha\Sig\cap W^s(\jT(e'))$ reads
\beq\label{eq:ws}
\Big\{\Big(0,s,0,\th_1,\sqrt{\pdemi-e'}\Big)\mid \abs{s}\leq \de,\ \th_1\in\T\Big\}
\eeq
while the intersection $\ha\Sig\cap W^u(\jT(e))$ takes the form
\beq\label{eq:wu}
\Big\{\Big(0,0,u,\th_1,\sqrt{\pdemi-e}+\mu\norm{\eta_\tau}_1\eta_\th'(\th_1)\Big)\mid \abs{u}\leq \de,\ \th_1\in\T\Big\}.
\eeq
Consider a point $\th_1^0\in\T$ and let $a\in\ha\Sig\cap W^u(\jT(e))$ be the point of coordinates 
$$
\Big(0,0,0,\th_1^0,\sqrt{\pdemi-e}+\mu\norm{\eta_\tau}_1\eta_\th'(\th_1^0)\Big)
$$
in the chart $(\tau,s,u,\th_1,r_1)$,  and let $a^\al\in\jT(e)$ be its $\al$--limit point under the flow of $H_n$. Then
$$
a\in \Sig\cap W^u(\jT(e))\cap W^s(\jT(e')),
$$
provided that
\beq\label{eq:difen}
e'=\pdemi-\Big(\sqrt{\pdemi-e}+\mu\norm{\eta_\tau}_1\eta_\th'(\th_1^0)\Big)^2.
\eeq
Now, due to the form of the vector field (\ref{eq:vectorbox}), one easily checks (using the global invariance
of the unstable foliation of $W^u(\jC)$ under the flow $\Phi_t^{H_n}$) that if $\ga$ is the tangent vector
to $W^{uu}(a^\al)$ at the point $a$, then the $e$ component of $\ga$ in the flow box coordinates is nonzero
as soon as $\th_1$ is in the support of $\eta_\th$.
So the tori $\jT(e)$ and $\jT(e')$ admit a heteroclinic connection as soon as there exists $\th_1^0$
satisfiying (\ref{eq:difen}), which is obviously true when $\abs{e-e'}$ is small enough (depending on $\mu$
and on the variation of $\eta_\th'$), and the previous remark morover proves that the weak splitting condition
is satisfied.
\end{proof}

It remains to examine the case of the singular cylinders. The previous lemma still applies to the three
regular subannuli $\sA_\bu^>$, $\sA_\bu^<$ and $\sA_\bu^0$, so it only remains to consider the neighborhood
of the critical energy $\ha e$.

\begin{lemme}\label{lem:weaktran2} Let $\jC$ be one of the two singular cylinders $\jC_\bu^\pm$. Let $\nu>0$ be fixed.
Then there exists an open interval $I^*$ containing $\ha e$ and $f\in C^{\ka-1}(\A^3)$
$$
\norm{f}_{C^{\ka-1}(\A^3)}\leq \nu
$$
such that each pair $(\jT^<(e),\jT^0(e'))$ or $(\jT^>(e),\jT^0(e'))$
with $e\in I^*$ and $e'\in I^*$ satisfies the weak splitting condition for the system $H_n=\H_n+f$.
\end{lemme}

\begin{proof} 
In fact the same considerations as in the previous lemma apply, thanks to the existence of transverse
homoclinic connections for the homoclinic orbits to the fixed point $O$ of $C_U$ (see Definition~\ref{def:singannuli}).
The singular annulus therefore admits a $C^1$ transverse homoclinic submanifolds $S^>,S^<$ in the neighborhood of each 
of the previous homoclinic connections. These submanifolds are almost everywhere of class $C^{\ka-1}$.
This enables one to find an interval $I^*$ and associated $C^1$ sections $\Sig^>,\Sig^<$ (almost everywhere of class $C^{\ka-1}$)
as above. One readily sees that the ``singular $C^1$ locus'' causes no trouble and the same arguments 
as above yield the existence of $f$, with controlled $C^{\ka-1}$ norm, for which the system $H_n$ satisfies our claim.
\end{proof}

\begin{cor}\label{cor:weaktran2}
Given $n\geq1$, there exists $f_n^{(1)}\in C^{\ka-1}(\A^3)$, with $\norm{f_n^{(1)}}_{C^{\ka-1}(\A^3)}\leq \tfrac{1}{2n}$,
 such that each $\jC\in\cyl(\H_n+f_n^{(1)})$
satisfies Condition (T).
\end{cor}

\begin{proof}
We apply the previous lemma inductively, after a preliminary ordering of all subintervals 
$(I_*(k))_{1\leq k\leq k^*}$  attached with all regular cylinders in $\cyl(\H_n)$ and a choice of pairwise disjoint 
attached  sections $\Sig$  and homoclinic curves $\sig$ (which is obviously possible thanks to the
structure of the set of annuli of $C_U$). Using the possibility to choose the support of the 
function $f$ in Lemma \ref{lem:weaktran1} inside an arbitrary neighborhood of $\sig$, we can therefore 
obtain a finite family of perturbations $f^k$, $1\leq k\leq k^*$,  with pairwise disjoint supports, such that the sum 
$f^{(1)}_n=\sum_k f^k$ satisfies the two claims of our statement since its norm is just the supremum of the
individual norms.  
\end{proof}


\subsection{Perturbation and Condition (S) for chains}
So far we have constructed a perturbed Hamiltonian $\H_n+f_n^{(1)}$ such that each cylinder of the family 
$\cyl(\H_n+f_n^{(1)})$ satisfies Condition (T). It remains now to add a new (and smaller) perturbation term 
to ensure that  the pairs of tori located in consecutive cylinders of the associated chain satisfy the weak splitting condition.
We begin with a classical lemma on the existence of heteroclinic intersections for tori with the same 
homology.

\begin{lemme}\label{lem:heter}
Set $\chain(\H_n+f_n^{(1)})=(\jC_k)_{1\leq k\leq k^*}$. 
Then for  $1\leq k\leq k^*-1$, there are tori $\jT_k\subset \jC_k$ and 
$\jT_{k+1}\subset\jC_{k+1}$ (of the family $\tori(\H_n+f_n^{(1)})$) which admit a heteroclinic connection.
\end{lemme}

\begin{proof} Let us begin with the unperturbed situation generated by $\H_n$.
Fix two consecutive (regular or singular) cylinders $\jC_k$ and $\jC_{k+1}$, associated with 
annuli $\sA_k$ and $\sA_{k+1}$. Then there exists an energy $e$ for which the  periodic
orbits $\sA_k\cap C_U\inv(e)$ and $\sA_{k+1}\cap C_U\inv(e)$ admit a transverse heteroclinic orbit, and 
this situation persists in a neighborhood of $e$ (recall that energy intervals over which our annuli are defined
admit small overlapping domains). As a consequence, 
as above, there exists a transverse section $\ha\Sig\subset\H_n\inv(\pdemi)$, endowed with 
symplectic coordinates  $(s,u,\th_1,r_1)$, such that $W^u(\jC_k)\cap \ha\Sig$ and $W^s(\jC_{k+1})\cap \ha\Sig$ 
read $\{u=0\}$ and $\{s=0\}$. The subset $\{u=s=0\}$ is the (local) intersection with $\Sig$ of a 
manifold of heteroclinic orbits between $\jC_k$ and $\jC_{k+1}$. This manifold $\jA$ is symplectic and 
diffeomorphic to $\T\times I$, where $I$ is some (small) open interval. The invariant manifolds $W^u(\jT_k(e))$ 
and $W^s(\jT_{k+1}(e))$ intersect $\jA$ along an essential circle $\Big\{r_1=\sqrt{2(\pdemi-e)}\Big\}$. 

\vskip1mm 

Now for $n$ large enough the perturbed situation for $\H_n+f_n^{(1)}$ is only a slight distorsion of the previous one. One can 
still find a  section $\Sig$ with coordinates $(s,u,\th_1,r_1)$ in which $W^u(\jC_k)\cap \Sig$ and 
$W^s(\jC_{k+1})\cap \Sig$  have the same equations as above and so intersect along the slightly perturbed annulus 
$\jA'$ with equation $u=s=0$ in the new coordinates (all this being  deduced from the various transversality 
properties). Again, by transversality,  $W^u(\jT_k(e))\cap \jA$  and $W^s(\jT_{k+1}(e))\cap \jA$  are  
embedded essential circles but they do not coincide any longer (in general). 

\vskip1mm 

However, it is easy to see that they still intersect each other, using the fact that the coordinates $(\th_1,r_1)$ 
are exact symplectic on $\jA$ together with the Lagrangian character of the invariant manifolds 
$W^u(\jT_k(e))$ and $W^s(\jT_{k+1}(e))$ (see \cite{LMS03} for more details). Indeed, since the 
tori $\jT_k(e)$ and  $\jT_{k+1}(e)$ are left unchanged when the perturbation is added, the intersections 
$C_k=W^u(\jT_k(e))\cap \jA$  and  $C_{k+1}=W^s(\jT_{k+1}(e))\cap \jA$ have the same homology in 
$\jA'$, meaning that the symplectic area between them vanishes. This comes from the fact that this assertion 
is trivially true in the unperturbed situation along with  the Lagrangian character of  $W^u(\jT_k(e))$  and 
$W^s(\jT_{k+1}(e))$. This proves our claim.
\end{proof}

Our next lemma will enable us to complete the proof of Proposition \ref{prop:condT}

\begin{lemme} For $n\geq n_0$ large enough,  there exists a function $f_n\in C^{\ka-1}(\A^3)$ with support
contained in the complement of $\cup_{1\leq k\leq k^*}\jC_k$, with 
$\norm{f_n}_{C^{\ka-1}(\A^3)}\leq \tfrac{1}{n}$, such that the chain  $(\jC_k)_{1\leq k\leq k^*}$
for the system $H_n:=\H_n+f_n$ satisfies Condition (S).
\end{lemme}

\begin{proof} The proof is similar and even simpler than that of Lemma \ref{lem:weaktran1}. With the notation
of Lemma \ref{lem:heter}, if the circles $C_k$ and $C_{k+1}$ intersect transversely in $\jA$, there is obviously
nothing to do. Now if they intersect tangentially, one constructs a flow-box as in Section~\ref{sec:flowbox}
and again uses a perturbation of the form
$$
\ell_n\circ\ha\chi_n(\tau,\e,s,u,\th_1,r_1)=\mu\eta_\tau(\tau)\eta_\th(\th_1).
$$
The support of $\ell_n$ can be chosen arbitrarilly small, and its norm is controlled by means of the constant
$\mu$. In particular, it can be chosen small enough to preserve the Condition (T) for all cylinders. One
can therefore proceed by induction as above, using now the natural ordering of the heteroclinically
connected pairs of tori inside consecutive cylinders of the chain.
This proves the existence of a finite family of functions $\ell_n^j$, with controlled supports and norms, such
that $f_n=f_n^{(1)}+\sum_j\ell_n^j$ fulfills our claims.
\end{proof}


\section{Diffusion orbits and proof of Theorem~\ref{thm:main}}
We first recall the $\lambda$--lemma of \cite{S13} in a version adapted to our present setting and state an
abstract shadowing  result for chains of cylinders. We then apply this result to prove the main theorem of 
this paper.


\subsection{Shadowing orbits along chains of minimal sets}

The $\lambda$-lemma in \cite{S13} requires the existence of the ``straightening neighborhood" (Proposition~B) 
for the cylinders.
In the case of general normally  hyperbolic manifolds such results need abstract additional assumptions, 
but here we will take advantage of the very simple geometric structure of the problem.

\paraga Let us begin with a straightening result in the neighborhood of the annuli. Let $U\in\jU$ be fixed.

\begin{lemme} Let $\sA$ be an annulus defined over $I$ for $C_U$. Then there exist a neighborhood
$\jO$ of $\sA$, an interval $\ha I$ containing $I$ and a symplectic diffeomorphism 
$\Psi: \T\times \ha I\times B\to \jO$, where $B=[-\al,\al]^2$ is a ball  in $\R^2$, such that
$\sA=\Psi(\T\times I\times \{0\})$ and
the composed Hamiltonian $\bC=C_U\circ \Psi$ takes the form
\beq\label{eq:hambC}
\bC(\ph,\rho,u,s)=\bC_0(\rho)+\la(\ph,\rho)\,us+\bC_3(\ph,\rho,u,s)
\eeq
with
\beq\label{eq:vanishbC}
\bC_3(\ph,\rho,0,0)=0,\qquad D\bC_3(\ph,\rho,0,0)=0,\qquad D^2\bC_3(\ph,\rho,0,0)=0.
\eeq
In particular, the local stable and unstable manifolds $W^{s,u}_{\ell}(\sA)$ together with the local stable and unstable manifolds
$W^{ss,uu}_{\ell}(x)$ for $x\in\sA$ are straightened in these coordinates and read:
$$
\Psi\inv\big(W^s_{\ell}(\sA)\big)=\{u=0\},\qquad \Psi\inv\big(W^u_{\ell}(\sA)\big)=\{s=0\},
$$
$$
\Psi\inv\big(W^{ss}_{\ell}(x)\big)=\{(\ph,\rho, s, 0)\mid s\in [-\al,\al]\},\quad \Psi\inv\big(W^{uu}_{\ell}(x)\big)=\{(\ph,\rho,0,u)\mid u\in[-\al,\al]\},
$$
where $(\ph,\rho)$ is defined by $\Psi(x)=(\ph,\rho,0,0)$.
\end{lemme}

The proof is a simple application of the Moser isotopy lemma. One
proves indeed the straightening result first and deduces the normal form from the structure of the Hamiltonian
system in such a neighborhood. The previous lemma yields the following straightening 
result in the neighborhood of the cylinders of $\cyl(H_n)$. 

\begin{lemme} \label{lem:straight} 
Let $\jC$ be a cylinder of the family $\cyl(H_n)$ and let $\sA$ be the associated annulus, 
defined over $I$. Let $\jO$ and $\Psi$ be defined as in the previous lemma.  Then, up to shrinking $B$ if necessary,
the product diffeomorphism  
$$
\ha\Psi={\rm Id}_\A\times \Psi\ :\  \A \times \T\times \ha I\times B\longrightarrow \A\times \jO
$$
is symplectic and satisfies
$$
H_n\circ\ha\Psi(\th_1,r_1,\ph,\rho,s,u)=\pdemi r_1^2+\bC_0(\rho)+O_2(s,u).
$$
\end{lemme}

\begin{proof} This is an immediate consequence of the fact that if $B$ is small enough, the neighborhood 
$\A\times\jO$ and the support  of $f_n$ are disjoint, so that $(H_n)_{\A\times\jO}=(\H_n)_{\A\times\jO}$. 
The claim then follows from the  previous lemma.
\end{proof}

Note that $\jC$ is then the set of all $\ha\Psi(\th_1,r_1,\ph,\rho,0,0)$ such that 
$$
\pdemi r_1^2+\bC_0(\rho)=\pdemi.
$$

The $\la$-lemma proved in \cite{S13} was stated in the framework of symplectic diffeomorphisms and 
normally  hyperbolic invariant submanifolds in a symplectic manifold. We therefore need to adapt it to
the present context, since  the cylinders $\jC$ are not normally hyperbolic in $\A^3$, but rather in 
$H_n\inv(\pdemi)$. The simplest way to  overcome this (easy) problem is to apply the lemma to the full 
normally hyperbolic manifold  $\jN=\ha\Psi(\A\times\T\times\ha I\times\{0\})$ (with the notation of the 
previous lemma) and the symplectic  diffeomorphism $\Phi^{H_n}$ (the time-one map). This is made possible by the previous 
straightening result (see \cite{S13} for a proof, the lack of compactness obviously causes no trouble here, 
due to the preservation of energy and the fact that  $\jC$ is relatively compact). We set $\Phi=\Phi^{H_n}$.

\vskip3mm

\noindent {\bf The $\la$-lemma.} {\em Let $\jC\in\cyl(H_n)$ be a cylinder at energy $\pdemi$ for 
the Hamiltonian  system  $H_n$ and let $\jN$ be the normally hyperbolic manifold of $\A^3$ defined 
above. Let $\De$ be a $1$--dimensional submanifold of $\A^3$ which transversely intersects
$W^s(\jN)$ at some point~$a$. Then $\Phi^n(\De)$ converges to the unstable leaf  
$W^{uu}\big(\Phi^n(\ell(a))\big)$ in the $C^0$ compact open topology, where $\ell(a)$ is the unique 
element of $\jC$ such that the point $a$ belongs to the stable leaf $W^{ss}(\ell(a))$.}

\vskip3mm

Let us make clear the notion of convergence used here (see \cite{S13} for details). The simplest way
to define it is to use Lemma \ref{lem:straight}. In the neighborhood $\A\times\jO$ and relatively to the
previous coordinates, if $x\sim(\th_1,r_1,\ph,\rho,0,0)\in\jC$, the unstable leaf $W^{uu}(x)$
reads
$$
W^{uu}(x)=\{(\th_1,r_1,\ph,\rho,0,u)\mid u\in\,[-\al,\al]\}.
$$
The first result in \cite{S13} (Theorem~1) is that for $n$ large enough, the connected  component $\De^n$
of $\Phi^n(a)$ in $\Phi(\De^{n-1})\cap (\A\times \jO)$ is a {\em graph} over the unstable direction, that is, 
it admits the equation
$$
\De^n=\Big\{\big(\th_1^n(u),r_1^n(u),\ph^n(u),\rho^n(u),s^n(u),u\big)\mid u\in\,]-\ov u,\ov u[\Big\}.
$$
The convergence statement then just says that 
$$
\norm{(\th_1^n(u),r_1^n(u),\ph^n(u),\rho^n(u),s^n(u))-(\th_1^n(0),r_1^n(0),\ph^n(0),\rho^n(0),0)}\to 0
$$
uniformly in $u$ when $n$ tends to $+\infty$, where $(\th_1^n(0),r_1^n(0),\ph^n(0),\rho^n(0),0,0) \sim \Phi^n(x)$. Of course one then gets more global formulation by using the
defintion of $W^u(\jC)$ as the union of the images by $\Phi$ of the local unstable manifold. Note that
this is only a $C^0$-convergence while a stronger $C^1$-convergence result was proved in  \cite{S13}.
The same definitions apply to the following case.

\vskip3mm

\begin{cor} Let $\jC\in\cyl(H_n)$. 
Let $\De$ be a $1$--dimensional submanifold of $H_n\inv(\pdemi)$ which transversely intersects
$W^s(\jC)$ in $H_n\inv(\pdemi)$ at some point $a$. Then $\Phi^n(\De)$ converges to the unstable 
leaf $W^{uu}\big(\Phi^n(\ell(a))\big)$ in the $C^0$ compact open topology.
\end{cor}

\begin{proof} 
Observe that the fact that $\De$ intersects $W^s(\jC)$ transversely in $H_n\inv(\pdemi)$
implies that $\De$ transversely intersects $W^s(\jN)$. Then apply the $\la$-lemma and use the invariance
of energy.
\end{proof}

\paraga We can now state the shadowing result proved in \cite{S13}. The method is reminiscent of that of \cite{BT99}.

\begin{prop}\label{prop:shadow}\emph{\textbf {[Shadowing lemma].}} 
 Set $\chain(H_n)=(\jC^i)_{1\leq i\leq i^*(n)}$. For $1\leq i\leq i^*$, let 
$(\jT_j^i)_{1\leq j\leq j^*_i}$  be a family of dynamically minimal invariant tori contained in 
$\jC^i$, such that 
\begin{itemize}
\item for $1\leq j\leq j^*_i-1$, there exists $a_j^i\in \jT_j^i$ such that $W^{uu}(a_j^i)$  intersects 
$W^s(\jC^i)$ transversely in  $H_n\inv(\pdemi)$, at some point contained in $W^s(\jT_{j+1}^i)$,
\item for $1\leq i\leq i^*-1$,  there exists $a_{j^*_i}^i\in \jT_{j^*_i}^i$ such that $W^{uu}(a_{j^*_i}^i)$  
intersects  $W^s(\jC^{i+1})$ transversely in  $H_n\inv(\pdemi)$, at some point 
contained in $W^s(\jT_1^{i+1})$. 
\end{itemize}
Then, for each $\rho>0$, there 
exists an orbit $\Ga$ at energy $\pdemi$ of $H_n$ which intersects each $\rho$--neighborhood
$\jV_\rho(\jT_j^i)$, for $1\leq i\leq i^*$ and $1\leq j\leq j^*_i$.
\end{prop}

%
%


\subsection{Asymptotic density: proof of Theorem~\ref{thm:main}}
It only remains now to gather the results of the previous sections and apply the previous shadowing lemma
to the chain of cylinders $\chain(H_n)$ and a suitable family of minimal tori inside. Fix $\de>0$. 
Given $n\geq 1$, we set as above  $\chain(H_n)=(\jC^i)_{1\leq i\leq i^*(n)}$.

\vskip1mm$\bu$ There exists $n_0$ such that for $n\geq n_0$, the union of the lines 
$(\S_{k_\ell})_{1\leq \ell\leq n}$ is $\de/4$--dense in $\S$.

\vskip1mm$\bu$ There exists $n_1\geq n_0$ such that for $n\geq n_1$, the sphere $\S$
is $\de/4$--dense in $H_n\inv(\pdemi)$, and therefore the union of the lines 
$(\S_{k_\ell})_{1\leq \ell\leq n}$ is $\de/2$--dense in $H_n\inv(\pdemi)$.

\vskip1mm$\bu$ By construction, there exists $n_2\geq n_1$ such that for $n\geq n_2$, 
$$
\dd\Big(\bigcup_{1\leq i\leq i^*(n)}\Pi(\jC^i),\bigcup_{1\leq \ell \leq n} \S_{k_\ell}\Big)\leq \de/6,
$$
where $\dd$ is the Hausdorff distance in $\R^3$.
This shows that $\bigcup_{1\leq i\leq i^*(n)}\Pi(\jC^i)$ is $\de/6$--dense in 
$\bigcup_{1\leq \ell \leq n} \S_{k_\ell}$.

\vskip1mm $\bu$ By density of the minimal tori in the cylinders (see Proposition \ref{prop:cyl}), 
and since the  chains satisfy Condition (S), for each $n\geq n_2$, one can exhibit a family of 
minimal tori $(\jT_j^i)$  satisfying the  assumptions of the shadowing lemma and such that the union 
$\cup_{i,j}\Pi(\jT_j^i)$  is $\de/6$--dense in $\cup_{1\leq i\leq i^*(n)}\Pi(\jC^i)$.

\vskip1mm$\bu$ Proposition~\ref{prop:shadow}, applied with $\rho=\de/6$, shows the existence of an orbit 
of $H_n$ whose projection is $\de/6$--dense in $\cup_{i,j}\Pi(\jT_j^i)$ and therefore
$\de/2$--dense in $\bigcup_{1\leq \ell \leq n} \S_{k_\ell}$, so also $\de$--dense in $H_n\inv(\pdemi)$.
This concludes the proof of Theorem~\ref{thm:main}.

%

\end{document}